\documentclass[12pt,leqno]{article}
\usepackage[shortlabels]{enumitem}
\usepackage[utf8]{inputenc} 
\usepackage[T1]{fontenc}
\usepackage{preamble}

\title{Uniform RC-positivity of direct image bundles \thanks{Mathematics Subject Classification 32U05, 32L05, 32J25. \newline Keywords: RC-positivity, uniform RC-positivity, direct images.}
}
\author{Kuang-Ru Wu}

\usepackage{fullpage}
\usepackage{multicol}

\begin{document}

\date{}

\parskip=6pt

\maketitle

\begin{abstract}
The concept of RC-positivity and uniform RC-positivity is introduced by Xiaokui Yang to solve a conjecture of Yau on projectivity and rational connectedness of a compact K\"ahler manifold with positive holomorphic sectional curvature. Some main theorems in Yang's proof hold under a weaker condition called weak RC-positivity. It is therefore natural to ask if (uniform) weak RC-positivity implies (uniform) RC-positivity. Another motivation for studying this problem is to understand the relation between rational connectedness of $X$ and (uniform) RC-positivity of the holomorphic tangent bundle $TX$. 

In this paper, we obtain results in this direction. In particular, we show that if a vector bundle $E$ is uniformly weakly RC-positive, then $S^kE\otimes \det E$ is uniformly RC-positive for any $k\geq 0$, and $S^kE$ is uniformly RC-positive for $k$ large. We also discuss an approach that might lead to a solution to the question of whether weak RC-positivity of $E$ implies RC-positivity of $E$.

\end{abstract}

\section{Introduction}

In \cite{YangCamb}, Yang introduces the notion of RC-positivity as a differential geometric counterpart of rational connectedness.  RC-positivity plays a crucial role in Yang's proof of a conjecture of Yau: If a compact K\"ahler manifold has positive holomorphic sectional curvature, then the manifold is projective and rationally connected. A stronger notion called uniform RC-positivity is introduced by Yang in \cite{YangForum} which also can be used to prove the same conjecture of  Yau. For the semipositive case of Yau's conjecture, see \cite{HeierWong,MatsumuraPAMQ,MatsumuraAJM}.

Let us recall the definition of RC-positivity and uniform RC-positivity. Let $E$ be a holomorphic vector bundle of rank $r$ over a compact complex manifold $X$ of dimension $n$. Given a Hermitian metric $H$ on $E$, we denote the Chern curvature of $H$ by $\Theta^H$, which is an $\text{End}E$-valued $(1,1)$-form. We denote by $TX$ the holomorphic tangent bundle of $X$. For a vector $u\in E_t$ and a tangent vector $v\in T_tX$ with $t\in X$, we define the expression $$H(\Theta^H u,u)(v,\bar{v})$$ to be  $\sum_{j,k} H(\Theta^H_{j\bar{k}}u,u)v_j \bar{v}_k$ locally where we write the curvature $\Theta^H=\sum_{j,k}\Theta^H_{j\bar{k}}dt_j\wedge d\bar{t}_k$ and  $v=\sum_j v_j\partial/\partial t_j$. 

\begin{definition} 

A Hermitian metric $H$ on a holomorphic vector bundle $E\to X$ is called RC-positive if for any $t\in X$ and any nonzero $u\in E_t$, there is a nonzero tangent vector $v\in T_tX$ such that $H(\Theta^Hu,u)(v,\bar v)>0$. On the other hand, a Hermitian metric $H$ is called uniformly RC-positive if for any $t\in X$, there is a nonzero tangent vector $v\in T_tX$ such that for any nonzero $u\in E_t$, we have $H(\Theta^H u,u)(v,\bar v)>0$.

A holomorphic vector bundle $E\to X$ is called (uniformly) RC-positive if it admits a (uniformly) RC-positive Hermitian metric.
\end{definition}
It is clear that uniform RC-positivity implies RC-positivity. To motivate the definition of (uniform) weak RC-positivity, let us consider a Hermitian metric $H$ on $E$ and the induced metric $h$ on the line bundle $O_{P(E^*)}(1)$ over the projectivized bundle $P(E^*)$. By a standard computation (for example, see \cite[Formula (4.5)]{YangCamb}), we know that if $(E,H)$ is RC-positive, then the curvature $\Theta$ of $h$ is positive on every fiber and has at least $r$ positive eigenvalues at every point in $P(E^*)$. The existence of such a metric $h$ on $O_{P(E^*)}(1)$ is called weak RC-positivity of $E$  (\cite[Definition 3.3]{YangCamb}).

Similarly, using the same computation, we see that if $(E,H)$ is uniformly RC-positive, then the curvature $\Theta$ of the induced metric $h$ satisfies
\begin{enumerate}[label=\alph*.]
    \item\label{a} $\Theta$ is positive on every fiber. 
    \item $\Theta$ has at least $r$ positive eigenvalues at every point in $P(E^*)$.
    \item For any point $t\in X$, there exists a nonzero tangent vector $v\in T_tX$, such that $\Theta(\Tilde{v},\bar{\Tilde{v}})|_{(t,[\zeta])}>0$ for any lift $\Tilde{v}$ of $v$ to $T_{(t,[\zeta])}P(E^*)$.
\end{enumerate}
Following Yang, we call the existence of such a metric $h$ on $O_{P(E^*)}(1)$ uniform weak RC-positivity of $E$. Note that in the third condition, we consider the lifts to the tangent space $T_{(t,[\zeta])}P(E^*)$ for any point $[\zeta]$ in the fiber $P(E_t^*)$ not just one point $[\zeta]$. The second condition is implied by the first and the third, so we will omit it later on. Let us summarize the definition.
\begin{definition} The bundle $E$ is called weakly RC-positive if there exists a metric $h$ on $O_{P(E^*)}(1)$ with properties a and b. The bundle $E$ is called uniformly weakly RC-positive if there exists a metric $h$ on $O_{P(E^*)}(1)$ with properties a and c.
\end{definition}
In Yang's solution to Yau's conjecture, two main theorems \cite[Theroem 1.3 and Theorem 1.4]{YangCamb}, although formulated in terms of RC-positivity,  hold under weak RC-positivity. So, it is natural to ask if weak RC-positivity of $E$ implies RC-positivity of $E$ (\cite[Question 7.11]{YangCamb} and \cite[Problem 13]{Inayama}). This question has the same flavor as a conjecture of Griffiths \cite{Griff69}: If $E$ is ample, then $E$ is Griffiths positive. For the developments on the Griffiths conjecture, see \cite{Umemura,CampanaFlenner,Berndtsson09,MourouganeTaka,positivityandvanishingthmliu,liu2014curvatures,FengLiuWan,demailly2020hermitianyangmills,pingali2021note,finski2020monge,Finskichara,wu_2022,wupositivelyII,wuIII,Mazhang,lempert2024two,murakami2025analytic,wu2025mean}.

In this paper, we make some progress in this direction. In particular, we prove the following theorem regarding uniform RC-positivity.

\begin{theorem}\label{Thm for E}
    If $E$ is uniformly weakly RC-positive over a compact K\"ahler manifold $X$, then $S^kE\otimes \det E$ is uniformly RC-positive for any $k\geq 0$, and $S^kE$ is uniformly RC-positive for $k$ large.
\end{theorem}

 Another motivation for establishing Theorem \ref{Thm for E} is to understand the relation between rational connectedness of $X$ and (uniform) RC-positivity of the holomorphic tangent bundle $TX$. According to Yang \cite[Theorem 1.4]{YangCamb} and \cite[Theorem 1.3]{YangForum}, for a compact K\"ahler manifold $X^n$, if one of the following is true, then $X$ is projective and rationally connected.
\begin{enumerate}
     \item The holomorphic tangent bundle $TX$ is uniformly RC-positive.
     \item The exterior power $\wedge^p TX$ is RC-positive for $1\leq p\leq n$.
 \end{enumerate}
 One can ask if the converse is true (\cite[Problem 4.15]{YangForum}). A partial converse is proved in \cite[Theorem 1.4]{YangForum}: if $X$ is projective and rationally connected, then the line bundle $O_{\wedge^p TX}(-1)$ is RC-positive for $1\leq p \leq n$. So, Theorem \ref{Thm for E} can be viewed as a step towards this converse problem: constructing uniformly RC-positive Hermitian metrics out of metrics on the line bundle $O_{P(E^*)}(1)$.

We also prove a lemma (Lemma \ref{lemma abc} in Section \ref{section ?}) and discuss how a variant of this lemma might lead to a solution to the original question of Yang, namely, weak RC-positivity of $E$ implying RC positivity of $E$.

For the proof of Theorem \ref{Thm for E}, instead of the fibration $p:P(E^*)\to X$, we will work on a more general fibration and prove a general theorem which contains Theorem \ref{Thm for E}
as a special case. We consider a proper holomorphic surjection $p:\mathcal{X}^{n+m}\to Y^m$ between two complex manifolds with $\mathcal{X}$ K\"ahler, $Y$ compact, and the differential $dp$ surjective at every point. We denote the fibers $p^{-1}(t)$ by $\mathcal{X}_t$ for $t\in Y$. Let $(L,h)$ be a Hermitian line bundle over $\mathcal{X}$. Let $$V_t=H^0(\mathcal{X}_t, L|_{\mathcal{X}_t}\otimes K_{\mathcal{X}_t}).$$ We assume that $\dim V_t$ is independent of $t\in Y$. So, the direct image of the sheaf of sections of $L\otimes K_{\mathcal{X}/Y}$ is locally free by Grauert's direct image theorem, where $K_{\mathcal{X}/Y}$ is the relative canonical bundle. We denote by $V$ the associated vector bundle over $Y$. There is a naturally defined Hermitian metric $H$ on $V$. For $u$ in $V_t$ with $t\in Y$, 
\begin{equation}\label{metric}
  H(u,u):=\int_{\mathcal{X}_t}h(u,u).  
\end{equation}
We extend the metric $h$ to act on sections $u$ of $L|_{\mathcal{X}_t}\otimes K_{\mathcal{X}_t}$ so that $h(u,u)$ is an $(n,n)$-form on $\mathcal{X}_t$. In terms of local coordinates, if $u=u'\otimes e$ with $u'$ an $(n,0)$-form and $e$ a frame of $L|_{\mathcal{X}_t}$, then $h(u,u)=c_n u' \wedge \overline{u'} h(e,e)$ where $c_n=i^{n^2}$. Under this more general fibration, we can show 
\begin{theorem}\label{Thm 1}
    If the curvature $\Theta$ of $h$ is positive on every fiber, and for any point $t\in Y$, there exists a nonzero tangent vector $v\in T_tY$, such that $\Theta(\Tilde{v},\bar{\Tilde{v}})|_{(t,z)}>0$ for any lift $\Tilde{v}$ of $v$ to $T_{(t,z)}\mathcal{X}$, then the Hermitian bundle $(V,H)$ is uniformly RC-positive.
\end{theorem}

Actually, the precise statement we prove in Theorem \ref{Thm 1} is: For a fixed point $t_0 \in Y$, if the curvature $\Theta$ of $h$ is positive on the fiber $\mathcal{X}_{t_0}$ , and  there exists a nonzero tangent vector $v\in T_{t_0}Y$, such that $\Theta(\Tilde{v},\bar{\Tilde{v}})|_{(t_0,z)}>0$ for any lift $\Tilde{v}$ of $v$ to $T_{(t_0,z)}\mathcal{X}$, then the Hermitian bundle $(V,H)$ is uniformly RC-positive at $t_0$.

Now, we consider the fibration $p:P(E^*)\to X$, and we assume $X$ is K\"ahler to make sure $P(E^*)$ is K\"ahler (see \cite[Subsection 5.2]{wu2025mean}). Therefore, by Theorem \ref{Thm 1}, we have Theorem \ref{Thm for E}. Indeed, the vector bundle $V$ in Theorem \ref{Thm 1} is associated with the direct image of $L\otimes K_{\mathcal{X}/Y}$. In the present situation, the relative canonical bundle
$K_{P(E^*)/X}$ is isomorphic to $O_{P(E^*)}(-r)\otimes p^*\det E$. If we choose $O_{P(E^*)}(r+k)$ for the line bundle $L$, then $V$ is $S^k\otimes \det E$. On the other hand, if we choose $O_{P(E^*)}(k)\otimes K^{-1}_{P(E^*)/X}$ for $L$, then $V$ is $S^kE$ (we use an arbitrary metric $g$ on $K^{-1}_{P(E^*)/X}$, and the effect of $g$ can be absorbed by taking $k$ large). 

The proof of Theorem \ref{Thm 1} is an adaptation of \cite[Section 3]{wu2025mean}, but we still include the details for completeness  (the original argument is due to Berndtsson in \cite[Section 4]{Berndtsson09} and \cite[Section 2]{BoMathz}. See also \cite{CampanaCaoMihai}).

This paper is organized as follows. In Section \ref{section prelim}, we give a local expression for the assumption (uniform weak RC-positivity) in Theorem \ref{Thm 1} which will be used in the proof of the main theorem. In Section \ref{section proof}, we prove Theorem \ref{Thm 1}. In Section \ref{section ?}, we discuss
a characterization of weak RC-positivity and its possible application.

I would like to thank Shin-ichi Matsumura for bringing to my attention the question of RC-positivity and weak RC-positivity. I am grateful to L\'aszl\'o Lempert, Siarhei Finski, and Xiaokui Yang for their interest in the paper. Thanks are also due to the Erd\H{o}s Center, Budapest for the support.

\section{Preliminary}\label{section prelim}

Let $p:\mathcal{X}^{m+n}\to Y^m$ be the fibration in the introduction, and $(L,h)\to \mathcal{X}$ be the Hermitian line bundle. In this section, we will explain the meaning of the assumption (uniform weak RC-positivity) in Theorem \ref{Thm 1}.

Assume that the Hermitian metric $h$ on the line bundle $L$ has its curvature $\Theta$ positive on every fiber. For such an $h$, we can decompose its curvature $\Theta=\Theta_{\mathcal{H}}+\Theta_\mathcal{V}$ where $\Theta_\mathcal{H}$ is the horizontal component and  $\Theta_\mathcal{V}$ is the vertical component. The horizontal component $\Theta_\mathcal{H}$ can be viewed as an element in $C^{\infty}(\mathcal{X}, p^*(\wedge^{1,1}T^*Y))$.

We denote by $(t_1,\ldots,t_m)$ the local coordinates in $Y$ and by $(t_1,\ldots,t_m,z_1,\ldots, z_n)$ the local coordinates in $\mathcal{X}$, and assume that $h=e^{-\phi}$ locally. We write $\phi_{i\bar{j}}=\phi_{t_i\bar{t}_j}$, $\phi_{\lambda\bar{\mu}}=\phi_{z_\lambda \bar{z}_\mu}$, etc. Since $\Theta$ is positive on each fiber, the matrix $(\phi_{\lambda\bar{\mu}})$ is positive definite, and we denote its inverse by $(\phi^{\lambda\bar{\mu}})$. The horizontal component $\Theta_\mathcal{H}$ and the vertical component $\Theta_\mathcal{V}$ have the local expressions 
\begin{align}
 &\Theta_\mathcal{H}=\sum_{i,j}( \phi_{i\bar{j}}-\sum_{\lambda,\mu} \phi_{i\bar{\mu}}\phi^{\lambda \bar{\mu}}\phi_{\lambda \bar{j}})dt_i\wedge d\bar{t}_j\label{horizontal}\\  &\Theta_\mathcal{V}=\sum_{\lambda,\mu}\phi_{\lambda\bar{\mu}}\delta z_\lambda\wedge \delta \bar{z}_{\mu}\label{vertical}
\end{align}
where $\delta z_\lambda=dz_\lambda+\sum_{i,\mu}\phi^{\lambda\bar{\mu}}\phi_{i\bar{\mu}}dt_i$. That (\ref{horizontal}) and (\ref{vertical}) are independent of local coordinates is discussed in
\cite[Subsection 1.1]{FLWgeodesiceinstein}, and the horizontal component $\Theta_\mathcal{H}$ is called the 
geodesic curvature there.

Recall the assumption in Theorem \ref{Thm 1}: the curvature $\Theta$ of $h$ is positive on every fiber, and for any point $t\in Y$, there exists a nonzero tangent vector $v\in T_tY$, such that $\Theta(\Tilde{v},\bar{\Tilde{v}})|_{(t,z)}>0$ for any lift $\Tilde{v}$ of $v$ to $T_{(t,z)}\mathcal{X}$. We may assume the tangent vector $v$ is $\partial/\partial t_1 $ in the local coordinates above, and we denote the lift by $\Tilde{v}=\partial/\partial t_1+\sum_\lambda a_\lambda \partial/\partial z_\lambda $ with $a_\lambda\in\mathbb{C}$. We then have \begin{equation}
\begin{aligned}
\Theta(\Tilde{v},\bar{\Tilde{v}})&=\phi_{1\bar{1}}+\sum_\mu \phi_{1\bar{\mu}}\overline{a_\mu}+\sum_\lambda \phi_{\lambda\bar{1}}a_\lambda+\sum_{\lambda,\mu}\phi_{\lambda\bar{\mu}}a_\lambda\overline{a_\mu}\\&=\phi_{1\bar{1}}-\sum_{\lambda,\mu}\phi_{1\bar{\mu}}\phi^{\lambda\bar{\mu}}\phi_{\lambda\bar{1}}+\sum_{\lambda,\mu}\phi_{1\bar{\mu}}\phi^{\lambda\bar{\mu}}\phi_{\lambda\bar{1}}+\sum_\mu \phi_{1\bar{\mu}}\overline{a_\mu}+\sum_\lambda \phi_{\lambda\bar{1}}a_\lambda+\sum_{\lambda,\mu}\phi_{\lambda\bar{\mu}}a_\lambda\overline{a_\mu}\\&=\phi_{1\bar{1}}-\sum_{\lambda,\mu}\phi_{1\bar{\mu}}\phi^{\lambda\bar{\mu}}\phi_{\lambda\bar{1}}+\|\sqrt{A^{-1}}\phi_1+\sqrt{A}\bar{a}\|^2,
\end{aligned}   
\end{equation}
 where in the last equality we denote the matrix $(\phi_{\lambda\bar{\mu}})$ by $A$, the column vector $(\phi_{\lambda\bar{1}})$ by $\phi_1$, and the column vector $(a_\lambda)$ by $a$. 
Let us summarize the computation as a lemma. 
\begin{lemma}\label{lemma assumption}
 $\Theta(\Tilde{v},\bar{\Tilde{v}})>0$ for any lift $\Tilde{v}$ if and only if $\phi_{1\bar{1}}-\sum_{\lambda,\mu}\phi_{1\bar{\mu}}\phi^{\lambda\bar{\mu}}\phi_{\lambda\bar{1}}>0$, which is also equivalent to $\Theta_{\mathcal{H}}(\tilde{v},\bar{\Tilde{v}})>0$.   
\end{lemma} 
 
\section{Proof of Theorem \ref{Thm 1}}\label{section proof} 

Recall the setup: $p:\mathcal{X}^{m+n}\to Y^m$ is a proper holomorphic submersion with $\mathcal{X}$ K\"ahler and $Y$ compact, $(L,h)\to \mathcal{X}$ is a Hermitian line bundle, and $V\to Y$ is a holomorphic vector bundle associated with the direct image $p_*(L\otimes K_{\mathcal{X}/Y})$. The fibers of $V$ are $V_t=H^0(\mathcal{X}_t, L|_{\mathcal{X}_t}\otimes K_{\mathcal{X}_t}).$  The bundle $V$ carries a Hermitian metric $H(u,u)=\int_{\mathcal{X}_t}h(u,u)$.

A smooth local section $u$ of the bundle $V$ is represented by a smooth $(n,0)$-from with values in $L$ over $p^{-1}(W)$ for some open set $W$ in $Y$ such that the restriction to each fiber is holomorphic. Any representative of $u$ is denoted by $\mathbf{u}$. We use $(t_1,\dots, t_m)$ for local coordinates in $Y$. In general, we have $$\bar{\partial}\mathbf{u}=\sum_j d\bar{t}_j\wedge \nu_j+  
 \sum_j \eta_j\wedge dt_j$$ where $\eta_j$ are of bidegree $(n-1,1)$ and $\nu_j$ are of bidegree $(n,0)$ whose restrictions to fibers are holomorphic (thus $\nu_j$ define sections of the bundle $V$). The Hermitian holomorphic vector bundle $(V, H)$ admits the Chern connection $D=D'+D{''}$. The $(0,1)$-part $D^{''}$ is given by $D^{''}u=\sum_j \nu_j d\bar{t}_j$. Therefore, a section $u$ of $V$ is holomorphic if and only if $\bar{\partial}\mathbf{u}=  \sum \eta_j\wedge dt_j$. 
 
 For the $(1,0)$-part $D'$, we consider some local frame $e$ of $L$ and write $\mathbf{u}=u'\otimes e$ with $u'$ an $(n,0)$-form. Denote $h(e,e)=e^{-\phi}$ and define 
 \begin{equation}\label{phi}
     \partial^\phi \mathbf{u}:=(\partial^\phi u')\otimes e = e^\phi \partial(e^{-\phi}{u'})\otimes e 
 \end{equation}
which is an $(n+1,0)$-form with values in $L$. It is straightforward to check that (\ref{phi}) is independent of the choice of the frame $e$. We can write  $\partial^\phi \mathbf{u}=\sum_j dt_j\wedge \mu_j$ where $\mu_j$ are of bidegree $(n,0)$. If we denote by $P(\mu_j)$ the orthogonal projection of $\mu_j$ on the space of holomorphic forms on each fiber, then $D'u=\sum_j P(\mu_j)dt_j$.

Next, we are going to prove $(V,H)$ is uniformly RC-positive. Fix a point $t_0\in Y$. By the assumption in Theorem \ref{Thm 1}, there exists a nonzero tangent vector $v_0\in T_{t_0}Y$ such that for any lift $\Tilde{v}$ of $v_0$ to $T_{(t_0,z)}\mathcal{X}$, we have $\Theta(\Tilde{v},\bar{\Tilde{v}})>0$. So, the goal is to show that the curvature $\Theta^V$ of $(V,H)$ satisfies $H(\Theta^V u_0,u_0)(v_0,\bar{v}_0)>0$ for any nonzero vector $u_0$ in $V_{t_0}$.

We choose a coordinate system $(t_1,\ldots,t_m)$ around the point $t_0$ in $Y$ such that $v_0=\partial/\partial t_1$ at $t_0$. Consider a fixed $u_0\neq 0$ in $V_{t_0}$. A standard argument allows us to extend $u_0$ to a local holomorphic section $u$ of $V$ such that $D'u=0$ at $t_0$ and $u(t_0)=u_0\neq0$. A straightforward computation gives
\begin{equation}\label{standard}   \partial \bar{\partial} H(u,u)=- H(\Theta^V u,u) \text{ at } t_0. 
\end{equation}
On the other hand, if we let $\mathbf{u}$ be a representative of $u$ and write $\mathbf{u}=u'\otimes e$ with $u'$ an $(n,0)$-form and $e$ some local frame of $L$, then 
\begin{equation}
 H(u,u)=p_*(c_n u'\wedge \overline{u'} e^{-\phi})   
\end{equation}
where $e^{-\phi}=h(e,e)$ and $c_n=i^{n^2}$. According to \cite[Proposition 4.2]{Berndtsson09}, we can choose a representative $\mathbf{u}$ such that in $\bar{\partial}\mathbf{u}=  \sum \eta_j\wedge dt_j$, the $\eta_j$ is primitive on $\mathcal{X}_{t_0}$. Moreover, 
$\partial^\phi u'=0$ at $t_0$. After using such a representative, we obtain
\begin{equation}\label{4.4}
\partial\bar{\partial}H(u,u)
 = 
    -c_n p_* ( u'\wedge\overline{u'}\wedge \partial\bar{\partial} \phi  e^{-\phi})
    +
    (-1)^n c_np_*  (\bar{\partial}u'\wedge \overline{\bar{\partial}u'}e^{-\phi}) \text{ at }  t_0.    
\end{equation}
We apply the above $(1,1)$-form to the tangent vector $v_0=\partial/\partial t_1$ and get 
\begin{equation}\label{4.6}
\partial\bar{\partial}H(u,u)(v_0,\bar{v}_0)
 = 
    -c_n p_* ( u'\wedge\overline{u'}\wedge \partial\bar{\partial} \phi  e^{-\phi})(v_0,\bar{v}_0)
    +
    (-1)^n c_np_*  (\bar{\partial}u'\wedge \overline{\bar{\partial}u'}e^{-\phi}) (v_0,\bar{v}_0)\text{ at }  t_0.
\end{equation}
Because $\bar{\partial}\mathbf{u}=  \sum \eta_j\wedge dt_j$ and $\mathbf{u}=u'\otimes e$, we see $\sum \eta_j\wedge dt_j=\bar{\partial}\mathbf{u}=\bar{\partial}u'\otimes e $. If we write $\eta_j=\eta_j'\otimes e$, then $\bar{\partial}u'=\sum \eta_j'\wedge dt_j$. So the last term in (\ref{4.6}) is equal to \begin{equation}\label{4.7}
   (-1)^n c_n\int_{\mathcal{X}_{t_0}}(-1)^{n}\sum \eta'_j\wedge 
\overline{\eta'}_k \wedge dt_j\wedge d\bar{t}_k e^{-\phi}(v_0,\bar{v}_0)=  c_n \int_{\mathcal{X}_{t_0}} \eta'_1\wedge 
\overline{\eta'}_1  e^{-\phi}\leq 0;
\end{equation}
the last inequality is by the fact that the $\eta_1$ is primitive on $\mathcal{X}_{t_0}$.

 We claim that the middle term in (\ref{4.6}) is negative. For the $(n,0)$-form $u'$, we can write locally \begin{equation}
    u'=u_z dz+\sum_ {\lambda, j}u_{z_\lambda t_j}d\hat{z}_\lambda\wedge dt_j+(\text{terms with more than one $t_j$}), 
\end{equation}
where $dz=dz_1\wedge\dots \wedge dz_n$ and $d\hat{z}_\lambda$ means $dz_1\wedge \dots \wedge dz_n$ omitting $dz_\lambda$; here $u_z$ and $u_{z_\lambda t_j}$ simply mean the coefficients, not differentiation.  By a degree count, we see that $-c_n p_* ( u'\wedge\overline{u'}\wedge \partial\bar{\partial} \phi  e^{-\phi})$ is equal to 
\begin{align*}
-c_n\int_{\mathcal{X}_{t_0}}
e^{-\phi} \big(&\sum_{j,k} u_z dz\wedge \overline{u_zdz} \wedge\phi_{j\bar{k}} dt_j\wedge d\bar{t}_k  \\+&\sum_{\lambda,j,k}
u_z dz\wedge 
\overline{u_{z_\lambda t_k}d\hat{z}_\lambda\wedge dt_k}\wedge
\phi_{j\bar{\lambda}} dt_j\wedge d\bar{z}_\lambda \\+&\sum_{\lambda,j,k}
u_{z_\lambda t_j}d\hat{z}_\lambda\wedge dt_j\wedge \overline{u_zdz}\wedge\phi_{\lambda \bar{k}} dz_\lambda \wedge d\bar{t}_k 
\\+
&\sum_{\lambda,\mu,j,k} u_{z_\lambda t_j}d\hat{z}_\lambda\wedge dt_j\wedge \overline{u_{z_\mu t_k}d\hat{z}_\mu\wedge dt_k}\wedge\phi_{\lambda \bar{\mu}} dz_\lambda\wedge d\bar{z}_\mu \big)
\end{align*}
which can be organized as 
\begin{equation}\label{4.8}
\begin{aligned}
-c_n \sum_{j,k}\int_{\mathcal{X}_{t_0}} e^{-\phi}\big(|u_z|^2\phi_{j\bar{k}}+&\sum_{\lambda} (-1)^{n-\lambda+1}u_z \overline{u_{z_\lambda t_k}} \phi_{j\bar{\lambda}}+\sum_{\lambda} (-1)^{n-\lambda+1}\overline{u_z} u_{z_\lambda t_j}\phi_{\lambda\bar{k}}\\+&\sum_{\lambda,\mu} (-1)^{\lambda+\mu}u_{z_\lambda t_j} \overline{u_{z_\mu t_k}}\phi_{\lambda \bar{\mu}} \big) dz\wedge d\bar{z}\wedge dt_j\wedge d\bar{t}_k.
\end{aligned}    \end{equation}
Applying the above $(1,1)$-form to $(v_0,\bar{v}_0)$, we see that $-c_n p_* ( u'\wedge\overline{u'}\wedge \partial\bar{\partial} \phi  e^{-\phi})(v_0,\bar{v}_0)$ is equal to 
\begin{equation}
\begin{aligned}\label{4.9}
-c_n \int_{\mathcal{X}_{t_0}} e^{-\phi}\big(|u_z|^2\phi_{1\bar{1}}+&\sum_{\lambda} (-1)^{n-\lambda+1}u_z \overline{u_{z_\lambda t_1}} \phi_{1\bar{\lambda}}+\sum_{\lambda} (-1)^{n-\lambda+1}\overline{u_z} u_{z_\lambda t_1}\phi_{\lambda\bar{1}}\\+&\sum_{\lambda,\mu} (-1)^{\lambda+\mu}u_{z_\lambda t_1} \overline{u_{z_\mu t_1}}\phi_{\lambda \bar{\mu}} \big) dz\wedge d\bar{z}.
\end{aligned}
\end{equation}
If we denote the matrix $(\phi_{\lambda\bar{\mu}})$ by $A$, the column vector $(\phi_{\lambda\bar{1}})$ by $\phi_1$, and the column vector $((-1)^{n-\lambda+1}u_{z_\lambda t_1})$ by $B_1$, then the expression inside the big parenthesis in (\ref{4.9}) can be written as 
\begin{equation}\label{4.10}
|u_z|^2\phi_{1\bar{1}}-\sum_{\lambda,\mu}\phi_{\lambda\bar{1}}\phi^{\lambda\bar{\mu}}\phi_{1\bar{\mu}}|u_z|^2+\|\sqrt{A^{-1}}\phi_1\overline{u_z}+\sqrt{A}\overline{B_1}\|^2.
\end{equation}
Combining (\ref{4.9}) and (\ref{4.10}), we get
\begin{equation}\label{4.11}
\begin{aligned}
    &-c_n p_* ( u'\wedge\overline{u'}\wedge \partial\bar{\partial} \phi  e^{-\phi})(v_0,\bar{v}_0)\\=&-c_n\int_{\mathcal{X}_{t_0}}e^{-\phi}\big(|u_z|^2(\phi_{1\bar{1}}-\sum_{\lambda,\mu}\phi_{\lambda\bar{1}}\phi^{\lambda\bar{\mu}}\phi_{1\bar{\mu}})+\|\sqrt{A^{-1}}\phi_1\overline{u_z}+\sqrt{A}\overline{B_1}\|^2\big)dz\wedge d\bar{z}. 
\end{aligned}
\end{equation}
Meanwhile, by Lemma \ref{lemma assumption}, the assumption in Theorem \ref{Thm 1} implies that $\phi_{1\bar{1}}-\sum_{\lambda,\mu}\phi_{\lambda\bar{1}}\phi^{\lambda\bar{\mu}}\phi_{1\bar{\mu}}>0$ on the fiber $\mathcal{X}_{t_0}$, so the integrand in (\ref{4.11}) is semipositive on the fiber $\mathcal{X}_{t_0}$. Actually,  the integrand in (\ref{4.11}) is positive somewhere on the fiber $\mathcal{X}_{t_0}$. This is because $u(t_0)\neq 0$, $\mathbf{u}|_{\mathcal{X}_{t_0}} $ cannot be identically zero, hence $u_z$ is nonzero somewhere.

As a consequence, from formula (\ref{4.11}), we deduce that $-c_n p_* ( u'\wedge\overline{u'}\wedge \partial\bar{\partial} \phi  e^{-\phi})(v_0,\bar{v}_0)$ is negative, as claimed. All in all, the right hand side in (\ref{4.6}) is negative, so $  \partial \bar{\partial} H(u,u)(v_0,\bar{v}_0)< 0$ at $t_0$. By (\ref{standard}), we get   $H(\Theta^V u_0,u_0)(v_0,\bar{v}_0)> 0$.

\section{Weak RC-positivity}\label{section ?}
We first give a characterization of weak RC-positivity. It is a variant of \cite[Lemma 18]{AndreottiVesentini}, \cite[Proposition 2.2]{YangCompos}, and \cite[Page 337]{demailly1997complex}. Let $p:\mathcal{X}^{m+n}\to Y^m$ be the fibration with the Hermitian line bundle $(L,h)\to \mathcal{X}$  in the introduction. We consider $\beta\in C^{\infty}(\mathcal{X}, p^*(\wedge^{1,1}T^*Y))$ and call $\beta$ positive if the matrix $(\beta_{i\bar{j}})$ in the local expression  $i\sum_{i,j}\beta_{i\bar{j}}dt_i\wedge d\bar{t}_j$ is positive. 

\begin{lemma}\label{lemma abc}
  Assume the curvature $\Theta$ of $h$ is positive on every fiber $\mathcal{X}_t$. For $1\leq k\leq m$, the following are equivalent. 
\begin{enumerate}[label=\Alph*.]
    \item\label{A} The curvature $\Theta$ has at least $n+k$ positive eigenvalues at every point in $\mathcal{X}$.
    \item\label{B} The horizontal component $\Theta_\mathcal{H}$ has at least $k$ positive eigenvalues at every point in $\mathcal{X}$.
    \item\label{C}  There exists a positive $\beta\in C^{\infty}(\mathcal{X}, p^*(\wedge^{1,1}T^*Y))$ such that the sum of any $m-k+1$ eigenvalues of $\Theta_\mathcal{H}$ with respect to $\beta$ is positive.
\end{enumerate}
Moreover, when $k=1$, statement C can be rephrased as $\Theta_\mathcal{H}\wedge \beta^{m-1}>0$ or equivalently $\Theta^{n+1}\wedge \beta^{m-1}>0$.
\end{lemma}

\begin{proof}
The equivalence between A and B can be seen by the decomposition $\Theta=\Theta_\mathcal{H}+\Theta_\mathcal{V}$ and formulas (\ref{horizontal}) and (\ref{vertical}). To prove that C implies B, we denote the eigenvalues  of $\Theta_\mathcal{H}$ with respect to $\beta$ by $\gamma_1\leq\ldots \leq \gamma_m$. Since $0<\gamma_1+\cdots+\gamma_{m-k+1}$, we see that $\gamma_{m-k+1}$ must be positive, so $\Theta_\mathcal{H}$ has at least $k$ positive eigenvalues.

To prove that B implies C, we first  fix a positive $\beta_0$ in $C^{\infty}(\mathcal{X}, p^*(\wedge^{1,1}T^*Y))$ and denote the eigenvalues  of $\Theta_\mathcal{H}$ with respect to $\beta_0$ by $\gamma_1\leq\ldots \leq \gamma_m$. We know that $\gamma_{m-k+1}>0$ since there are at least $k$ positive eigenvalues for $\Theta_\mathcal{H}$. We are going to construct $\beta$ using the arguments in \cite[Page 337]{demailly1997complex} (see also \cite[Lemma 18]{AndreottiVesentini} and \cite[Proposition 2.2]{YangCompos}).

Consider the positive numbers $A:=\inf_{\mathcal{X}}\gamma_{m-k+1}>0$,  $B:=\sup_{\mathcal{X}}\max_j |\gamma_j|>0$, and $\varepsilon:=1/(m-k+1)$. Let $\psi_\varepsilon$ be in $C^{\infty}(\mathbb{R},\mathbb{R})$ such that 
\begin{equation}
\psi_\varepsilon(x)=x \text{ for } x\geq A,\,\,\, \psi_\varepsilon(x)\geq x \text{ for } 0\leq x \leq A,\,\,\, \psi_\varepsilon(t)=B/\varepsilon \text{ for } x\leq 0.
\end{equation}
By \cite[Page 337, Lemma 5.2]{demailly1997complex}, $\psi_\varepsilon(\Theta_\mathcal{H})$ is in $C^{\infty}(\mathcal{X}, p^*(\wedge^{1,1}T^*Y))$ and positive. For a point $(t,z)\in \mathcal{X}$, if $\{\xi_1,\ldots, \xi_m\}$ is a  basis of $T^*_t Y$
such that  $\beta_0(t,z)=\sum \xi_j\wedge \bar{\xi}_j$, then
\begin{equation}
\Theta_\mathcal{H}(t,z)=\sum \gamma_j\xi_j\wedge \bar{\xi}_j \text{ and } \psi_\varepsilon(\Theta_\mathcal{H})(t,z)=\sum \psi_\varepsilon(\gamma_j) \xi_j\wedge \bar{\xi}_j.  
\end{equation} 
Therefore, the eigenvalues of $\Theta_\mathcal{H}$ with respect to $\psi_\varepsilon(\Theta_\mathcal{H})$ are $\gamma_j/\psi_{\varepsilon}(\gamma_j)$.  For $1\leq j\leq m$, 

\[
\left\{ 
\begin{array}{l}
\text{if } 0\leq \gamma_j, \text{ then } \gamma_j\leq \psi_\varepsilon(\gamma_j) \text{ and } 0\leq \gamma_j/\psi_\varepsilon(\gamma_j)\leq 1. \\
        \text{if } \gamma_j<0, \text{ then } \psi_\varepsilon(\gamma_j)=B/\varepsilon \text{ and } -\varepsilon\leq \gamma_j/(B/\varepsilon)\leq 0.
    \end{array}
\right.
\]
All in all, $-\varepsilon\leq \gamma_j/\psi_\varepsilon(\gamma_j)\leq 1$ for $1\leq j \leq m$. On the other hand, for $m-k+1\leq j\leq m$, we have $A\leq \gamma_j$, hence $\psi_\varepsilon(\gamma_j)=\gamma_j$ and $\gamma_j/\psi_\varepsilon(\gamma_j)=1$.

As a result, the sum of any $m-k+1$ eigenvalues of $\Theta_\mathcal{H}$
 with resect to $\psi_\varepsilon(\Theta_\mathcal{H})$ is at least $1-\varepsilon(m-k)=\varepsilon>0$ as we chose $\varepsilon=1/(m-k+1)$. By setting $\beta=\psi_\varepsilon(\Theta_\mathcal{H})$, the proof is complete.

For the moreover part, it is clear  that, when $k=1$, statement C can be rephrased as $\Theta_\mathcal{H}\wedge \beta^{m-1}>0$. For the equivalence to $\Theta^{n+1}\wedge \beta^{m-1}>0$, we will use the following formula:
\begin{equation}\label{wzw formula}
\Theta^{n+1}\wedge\beta^{m-1}
  =(n+1)! (m-1)!
 \sum_{i,j} \beta^{i\bar{j}}\det M(i\bar{j})
   \det(\beta)
  \big(\bigwedge^m_{k=1} i dt_k\wedge d\Bar{t}_k\wedge \bigwedge^n_{\lambda=1} i dz_\lambda\wedge d\Bar{z}_\lambda\big). 
\end{equation}
Here $M(i\bar{j})$ for fixed $i$ and $j$ is a matrix defined by
\begin{equation}\label{matrix}
M(i\bar{j}):=\left (
\begin{array}{cccc}
\phi_{t_i\bar{t}_j}
& \phi_{t_i\bar{z}_1} & \cdots &  \phi_{t_i\Bar{z}_n}\\
\phi_{z_1\bar{t}_j} & \phi_{z_1\bar{z}_1} & \cdots & \phi_{z_1\bar{z}_n}\\
 \vdots & \vdots & \ddots & \vdots \\
\phi_{z_n\bar{t}_j}& \phi_{z_n\bar{z}_1} &\cdots & \phi_{z_n\bar{z}_n} 
\end{array}
\right )
\end{equation}
with respect to local coordinates $(t_1,\dots,t_m)$ in $Y$ and local coordinates $(t_1,\dots,t_m,z_1,\dots, z_n)$ in $\mathcal{X}$, and $\phi$ is a local weight for the metric $h$, namely $h=e^{-\phi}$. Formula (\ref{wzw formula}) is a variant of formula (3) in \cite{wu2025mean}. To prove formula (\ref{wzw formula}), it suffices to consider a fixed point $(t,z)$ in $\mathcal{X}$. Let $\alpha$ be a Hermitian metric on $Y$ such that $p^*\alpha=\beta$ at $(t,z)$. We then apply formula (3) in \cite{wu2025mean} to deduce formula (\ref{wzw formula}). (Formula (\ref{wzw formula}) is related to the Wess--Zumino--Witten equation, see \cite{wu23,wu2024potential,finski2024WZW,finski2024lower}).

According to Schur's formula, we have
\begin{equation}\label{Schur}
 \phi_{i\bar{j}}-\sum_{\lambda,\mu} \phi_{i\bar{\mu}}\phi^{\lambda \bar{\mu}}\phi_{\lambda \bar{j}}=\frac{\det M(i\bar{j})}{\det (\phi_{\lambda\bar{\mu}})}.
\end{equation}
A reminder on the notation: $M(i\bar{j})$ for fixed $i$ and $j$ is itself a matrix given in (\ref{matrix}), but $(\phi_{\lambda\bar{\mu}})$ is a matrix with $(\lambda,\mu)$ entry equal to $\phi_{\lambda\bar{\mu}}$. By formula (\ref{wzw formula}), the sign of $\Theta^{n+1}\wedge \beta^{m-1}$ is decided by $\sum_{i,j}\beta^{i\bar{j}}\det M(i\bar{j})$ which has the same sign with 
\begin{equation}\label{sign}
    \sum_{i,j}\beta^{i\bar{j}} (\phi_{i\bar{j}}-\sum_{\lambda,\mu} \phi_{i\bar{\mu}}\phi^{\lambda \bar{\mu}}\phi_{\lambda \bar{j}})
\end{equation}
by formula (\ref{Schur}) and the fact $(\phi_{\lambda\bar{\mu}})$ is positive definite. But the sign of  $\Theta_{\mathcal{H}}\wedge \beta^{m-1}$ is also decided by (\ref{sign}) 
according to (\ref{horizontal}). All together, $\Theta^{n+1}\wedge \beta^{m-1}>0$ if and only if $\Theta_{\mathcal{H}}\wedge \beta^{m-1}>0$. \end{proof}

The case we care about most in this paper is when $k=1$ in Lemma \ref{lemma abc} because it corresponds to weak RC-positivity. The difficulty in proving a theorem like Theorem \ref{Thm for E} or Theorem \ref{Thm 1} for weak RC-positivity is that the $\beta$ in Lemma \ref{lemma abc} is on $\mathcal{X}$, so it does not quite fit into Berndtsson's computation, especially formula (\ref{4.11}).
So, we raise the question:
\begin{question}
  Is it possible to choose $\beta$ in Lemma \ref{lemma abc} so that $\beta=p^*\alpha$ for some Hermitian metric $\alpha$ on $Y$?    
\end{question}
This question is somewhat bold because if it is possible to choose $\beta=p^*\alpha$, then we can use \cite[Theorem 4]{wu2025mean} to deduce that if $E$ is weakly RC-positive, then $S^k\otimes \det E$ has positive mean curvature for $k\geq 0$. Moreover, it is even possible to use \cite[Theorem 5]{wu2025mean} to deduce that if $E$ is weakly RC-positive, then $E$ has positive mean curvature. Since positive mean curvature implies RC-positivity (\cite[Theorem 3.6]{YangCamb}), this would mean that RC-positivity, weak RC-positivity, and mean curvature positivity are all equivalent. Such an equivalence is conjectured for tangent bundle $TX$ in
\cite[Problems 4.15 and 4.17]{YangForum}.

%In order to make a Hermitian metric on $X$ out of $\psi_\varepsilon(\Theta_H)$, we consider 
%\begin{equation}
%\alpha_\varepsilon:= \sum   \frac{\int_{P(E_t^*)}\psi_\varepsilon(\gamma_j)\omega^{r-1}_t}{\int_{P(E_t^*)}\omega^{r-1}_t}  \xi_j\wedge \bar{\xi}_j,
%\end{equation}
%where $\omega_t$ is $\Theta|_{P(E^*_t)}$. Then $\alpha_\varepsilon$ is a Hermitian metric on $X$, and the eigenvalues of $\Theta_H$ with respect to $\alpha_\varepsilon$ are $\gamma_j/\int_{P(E_t^*)}\psi_\varepsilon(\gamma_j)\omega^{r-1}_t$.  

\bibliographystyle{amsalpha}
\bibliography{Dominion}

\textsc{Erdős Center, HUN-REN Rényi Institute, Reáltanoda utca 14, H-1053, Budapest, Hungary}

\texttt{\textbf{wuuuruuu@gmail.com}}

\end{document}